\documentclass[11pt]{article}
\usepackage{mathrsfs}
\usepackage{hyperref}
\usepackage{amsmath,amsthm,amssymb,cite}
\usepackage{diagbox}
\usepackage{epic,eepic}
\usepackage{graphicx}
\usepackage{soul,color}
\usepackage{cases}
\usepackage{subeqnarray}
\usepackage{verbatim}
\usepackage{enumerate}
\usepackage{multirow}
\usepackage{indentfirst}
\usepackage{cleveref}
\usepackage{algorithm}
\usepackage{algcompatible}
\usepackage{tikz}
\usetikzlibrary{shapes,arrows,chains}

\makeatletter

\newcommand{\Rmnum}[1]{\expandafter\@slowromancap\romannumeral #1@}
\makeatother
\newtheorem{theorem}{Theorem}[section]

\newtheorem{lemma}{Lemma}[section]

\allowdisplaybreaks[4]

\begin{document}
\title{On the growth factor upper bound for Aasen's algorithm \thanks{This work was supported in part by National Natural Science Foundation of China grant 11671105 and the CSC (grant 201606310121).}}
\author{Yuehua Feng\thanks{School of Mathematical Science, Xiamen University, China. E-mail: {\tt yhfeng@stu.xmu.edu.cn}.}
\and  Linzhang Lu\thanks{School of Mathematical Science, Guizhou Normal University \& School of Mathematical
Science, Xiamen University, China. E-mail: {\tt llz@gznu.edu.cn,}{\tt lzlu@xmu.edu.cn}.}
}
%
\maketitle
\begin{abstract}
 Aasen's algorithm factorizes a symmetric indefinite matrix $A$ as $A = P^TLTL^TP$, where $P$ is a permutation matrix, $L$ is unit lower triangular with its first column being the first column of the identity matrix, and $T$ is tridiagonal. In this note, we provide a growth factor upper bound for Aasen's algorithm which is much smaller than that given by Higham. We also show that the upper bound we have given is not tight when the matrix dimension is greater than or equal to $6$.
\end{abstract}

\medskip
{\small

}

\section{Introduction}

Aasen's algorithm \cite{aasen1971reduction} factorizes a symmetric indefinite matrix $A \in \mathbb{R}^{n \times n}$ as
$$PAP^T = LTL^T,$$
where $L$ is unit lower triangular with first column $e_1$ (the first column of the identity matrix), $T$ is symmetric tridiagonal, and $P$ is a permutation matrix chosen such that $\left|l_{ij}\right| \leqslant 1$. As it is well known that we can effectively solve symmetric indefinite linear systems $Ax=b$ when $A$ is a dense matrix, Aasen's algorithm has already been incorporated into LAPACK \cite{anderson1999lapack}. 

For Aasen's algorithm, Higham gave a growth factor upper bound, which is $4^{n-2}$. However, whether this bound is attainable for $n \geqslant 4$ is an open problem \cite{higham2002accuracy}. Cheng \cite{cheng1998symmetric} constructed an example to obtain the growth factor of $4$ for $n=3$, and reported experiments using direct search method in which he obtained growth of $7.99$ for $n=4$ and $14.61$ for $n=5$.

In this note,  we report that the growth factor upper bound for Aasen's algorithm is  $2^{n-1}$, which is much smaller than $4^{n-2}$ given by Higham. Moreover, we also show that the upper bound $2^{n-1}$ we have obtained is not tight when $n \geqslant 6$.

This paper is organized as follows. Section \ref{sec:upper bounds of entries of T} presents the upper bounds on entries of the tridiagonal matrix $T$ in Aasen's algorithm. Section \ref{sec:GWF for Aasen algorithm} reports the growth factor bound for Aasen's algorithm. We draw some conclusions in Section \ref{sec:conclusion}. Without loss of generality, we assume that $\max_{i,j}  \left| a_{i j} \right| = 1$ for the symmetric matrix $A =(a_{ij})_{i,\, j=1}^n \in \mathbb{R}^{n \times n}$ and $A \stackrel{def}{=} P \, A \, P^T$ for convenience.

\section{Upper bounds on entries of \texorpdfstring{$T$}{Lg}}\label{sec:upper bounds of entries of T}

We first present a lemma which gives upper bounds on entries of the factor matrix $T$.

\begin{lemma} \label{Le:bound for tij}
Let $A =(a_{ij})_{i,\, j=1}^n \in \mathbb{R}^{n \times n}$ be symmetric, $\max_{i,j}  \left| a_{i j} \right| = 1$ and suppose Aasen's algorithm produces
\begin{equation}
A \stackrel{def}{=} P \, A \, P^T= L \, T \, L^T,
\end{equation}
where $L=(l_{ij})_{i,\, j=1}^n$ is an unit lower triangular matrix with first column $e_1$, $\left| l_{i j} \right| \leqslant 1$ and $T=(t_{ij})_{i,\, j=1}^n$ is a symmetric tridiagonal matrix. Then the following inequalities hold:
\begin{subequations}
\begin{numcases}{}
\left| t_{11} \right| \leqslant 1, \quad \left| t_{21} \right| \leqslant 1, \quad \left| t_{22} \right| \leqslant 1, \quad  \left| l_{i2} \, t_{21} \right| \leqslant 1, \quad 3 \leqslant i \leqslant n, \label{ineq: bounds of bound is 1}\\
\left| l_{i, \, j-1} \, t_{j-1, \, j} + l_{ij} \, t_{jj} + l_{i, \, j+1} \, t_{j+1, \, j} \right| \leqslant 2 ^{j-2}, \qquad 2 \leqslant j < i \leqslant n, \label{ineq: bounds of bound is power(2,j-2)}\\
\left| l_{n, \, n-1} \, t_{n-1, \, n} + t_{nn} \right| \leqslant 2 ^{n-2}, \label{ineq: bounds of bound is power(2,n-2)}\\
\left| t_{i, \, i-1} \right| \leqslant 2^{i-2} , \qquad \left| t_{i i} \right| \leqslant 2^{i-1},  \qquad 3 \leqslant i \leqslant n. \label{ineq: bounds of bound of t(i,:)}
\end{numcases}
\end{subequations}
\end{lemma}
\begin{proof}
We proceed by induction on $n$. If $n = 3$, the result is trivial (see the example in Cheng \cite{cheng1998symmetric}). Assume that it holds true for $n = k$. Consider $n = k+1$. By direct calculation and $2 \leqslant q \leqslant k$,
\begin{subequations}
\begin{numcases}{}
 a_{k+1, \, 1} = l_{k+1, \, 2} \, t_{21}, \label{eq:a(k+1,1)} \\
 a_{k+1, \, q} = \sum_{j=2}^q l_{q j} \, (l_{k+1, \, j-1} \, t_{j-1, \, j} + l_{k+1, \, j} \, t_{jj} + l_{k+1, \, j+1} \, t_{j+1, \, j}),  \label{eq:a(k+1,q)}  \\
 a_{k+1, \, k+1} = \sum_{j=2}^k l_{k+1,\, j} \, (l_{k+1, \, j-1} \, t_{j-1, \, j} + l_{k+1, \, j} \, t_{jj} + l_{k+1, \, j+1} \, t_{j+1, \, j}) \notag \\
\qquad \qquad \quad + \, l_{k+1, \, k} \, t_{k, \, k+1} + t_{k+1,\, k+1}. \label{eq:a(k+1,k+1)}
\end{numcases}
\end{subequations}

Since $\max_{i,j}  \left| a_{i j} \right| = 1$, then
\begin{align*}
 \left| l_{k+1, \, 2} \, t_{21} \right| \leqslant 1.
\end{align*}

By \eqref{eq:a(k+1,q)} and $| l_{gh} | \leqslant 1 ~(1\leqslant g, \, h \leqslant n)$, then 
\begin{gather*}
 \left| l_{k+1, \, 2} \, t_{22} + l_{k+1, \, 3} \, t_{32} \right| \leqslant 1, \\
 \left| l_{k+1, \, 2} \, t_{23} + l_{k+1, \, 3} \, t_{33} + l_{k+1, \, 4} \, t_{43} \right| \leqslant 2, \\
 \left| l_{k+1, \, 3} \, t_{34} + l_{k+1, \, 4} \, t_{44} + l_{k+1, \, 5} \, t_{54} \right| \leqslant 2^2,  \\
 \vdots \\
 \left| l_{k+1, \, k-2} \, t_{k-2, \, k-1} + l_{k+1, \, k-1} \, t_{k-1, \, k-1} + l_{k+1, \, k} \, t_{k, \, k-1} \right| \leqslant 2^{k-3}, \\
 \left| l_{k+1, \, k-1} \, t_{k-1, \, k} + l_{k+1, \, k} \, t_{kk} + l_{k+1, \, k+1} \, t_{k+1, \, k} \right| \leqslant 2^{k-2}.
\end{gather*}

Furthermore, from \eqref{eq:a(k+1,k+1)}
\begin{equation} \label{ineq: bounds of t(k,k+1)t(k+1,k+1)}
\left| l_{k+1, \, k} \, t_{k, \, k+1} + t_{k+1, \, k+1} \right| \leqslant \sum_{j=2}^k 2^{j-2} +1 = 2^{k-1}.
\end{equation}

Rewrite \eqref{eq:a(k+1,q)} as
\begin{align*}
a_{k+1,  \, k} &= \sum_{j=2}^k l_{k j} \, (l_{k+1, \, j-1} \, t_{j-1, \, j} + l_{k+1, \, j} \, t_{jj} + l_{k+1, \, j+1} \, t_{j+1, \, j}) \\
& = \sum_{j=2}^{k-1} l_{k+1, \, j} \, (l_{k, \, j-1} \, t_{j, \, j-1} + l_{k j} \, t_{j j} + l_{k, \, j+1} \, t_{j, \, j+1}) \\
& + l_{k+1,\, k} \, (l_{k,\, k-1} \, t_{k, \, k-1} + t_{k k}) + t_{k+1, \, k}.
\end{align*}

Apply the induction hypothesis,
\begin{equation*}
\left|t_{k+1, \, k} \right| \leqslant \sum_{j=2}^{k-1} 2^{j-2} + 2^{k-2} + 1 = 2^{k-1},
\end{equation*}
by \eqref{ineq: bounds of t(k,k+1)t(k+1,k+1)},
\begin{equation*}
\left|t_{k+1, \, k+1} \right| \leqslant 2^k.
\end{equation*}

Therefore, when $n = k+1$ the lemma still holds true. We thus prove the lemma.
\end{proof}

\section{Growth factor bound for Aasen's algorithm}\label{sec:GWF for Aasen algorithm}

In this section, we  present our main result. From Lemma \ref{Le:bound for tij} and if $| t_{n,n} | = 2^{n-1}$, without loss of generality, we assume that $t_{n,n}  = 2^{n-1}$, then by \eqref{eq:a(k+1,k+1)} with $k=n-1$,
\begin{equation} \label{eq:t(n,n)+}
\left\{
\begin{aligned}
& l_{nj} \, (l_{n, \, j-1} \, t_{j-1, \, j} + l_{nj} \, t_{jj} + l_{n, \, j+1} \, t_{j+1, \, j}) = -2^{j-2},\quad 2 \leqslant j \leqslant n-1, \\
& l_{n, \, n-1} \, t_{n-1, \, n} = -2^{n-2}, \\
& t_{nn}  = 2^{n-1}.
\end{aligned}
\right.
\end{equation}

By \eqref{eq:t(n,n)+} and \eqref{ineq: bounds of bound is power(2,j-2)} - \eqref{ineq: bounds of bound of t(i,:)}, 
\begin{equation} \label{eq:|l(i,j)|=1}
l_{nj} = \pm 1, \qquad 2 \leqslant j \leqslant n-1.
\end{equation}

Substituting \eqref{eq:t(n,n)+} and \eqref{eq:|l(i,j)|=1} into \eqref{eq:a(k+1,q)} with $k=n-1$, we obtain
\begin{equation*}
\left| \sum_{j=2}^q \, l_{qj} \, l_{nj} \, 2^{j-2} \right| \leqslant 1, \qquad 2 \leqslant q \leqslant n-1,
\end{equation*}
thus
\begin{equation} \label{eq:l(q,j)l(n,j)l(n,q) = -1}
l_{qj} \, l_{nj} \, l_{nq} = -1, \qquad 2 \leqslant j < q \leqslant n-1,
\end{equation}
furthermore,
\begin{equation} \label{eq:l(q,j)l(i,j)l(i,q) = -1}
l_{qj} \, l_{ij} \, l_{iq} = -1, \qquad 2 \leqslant j < q < i \leqslant n-1.
\end{equation}

Next we will derive an upper bound on $t_{nn}$ from \eqref{eq:l(q,j)l(n,j)l(n,q) = -1} and \eqref{eq:l(q,j)l(i,j)l(i,q) = -1}.

 At first, we rewrite \eqref{eq:t(n,n)+} as
\begin{numcases}{}
l_{nj} \, (l_{n, \, j-1} \, t_{j-1, \, j} + l_{nj} \, t_{jj} + l_{n, \, j+1} \, t_{j+1, \, j}) = -2^{j-2} + \delta_{j-2},\quad 2 \leqslant j \leqslant n-1, \notag \\
l_{n, \, n-1} \, t_{n-1, \, n} = -2^{n-2}  + \delta_{n-2}, \label{eq:t(n,n)+ with delta}\\
t_{nn}  = 2^{n-1} - \delta, \notag
\end{numcases}
where $\delta = \sum_{j=0}^{n-2} \, \delta_j ~ \mbox{and} ~  0 \leqslant \delta_j \leqslant 2^{j+1}~ (0 \leqslant j \leqslant n-2)$. Hence the upper bound on $t_{nn}$ is equivalent to the following optimization problem
\begin{equation} \label{initial optimization problem}
\left\{
\begin{aligned}
& \delta = \min \, \sum_{j=0}^{n-2} \, \delta_j \\
& \mbox{s. t. } \quad 0 \leqslant \delta_j \leqslant 2^{j+1}, \quad 0 \leqslant j \leqslant n-2.
\end{aligned}
\right.
\end{equation}

In the following analysis, we will seek more constraints on $\delta_j$. Substituting \eqref{eq:l(q,j)l(n,j)l(n,q) = -1} and \eqref{eq:t(n,n)+ with delta} into \eqref{eq:a(k+1,q)},
\begin{equation} \label{ineq: bounds of all delta}
0 \leqslant \delta_{q-2} - \sum_{j=2}^{q-1} \delta_{j-2} \leqslant 2, \qquad 3 \leqslant q \leqslant n-1.
\end{equation}

From \eqref{eq:l(q,j)l(n,j)l(n,q) = -1} and \eqref{eq:t(n,n)+ with delta},
\begin{equation} \label{eq:t(n,n)+ with delta2}
\left\{
\begin{aligned}
& -l_{j,j-1} \, t_{j-1,j} + t_{jj} - l_{j+1,j} \, t_{j+1,j} = -2^{j-2} + \delta_{j-2},\quad 2 \leqslant j \leqslant n-2, \\
& -l_{n-1,n-2} \, t_{n-2,n-1} + t_{n-1,n-1} = 2^{n-3} + \delta_{n-3} - \delta_{n-2}, \\
& l_{n, \, n-1} \, t_{n-1, \, n} = -2^{n-2}  + \delta_{n-2}, \\
& t_{nn}  = 2^{n-1} - \delta.
\end{aligned}
\right.
\end{equation}

For $2 \leqslant q < i \leqslant n-1$, multiply $l_{iq}$ on both sides of \eqref{eq:a(k+1,q)} with $k=i-1$, then from \eqref{eq:l(q,j)l(i,j)l(i,q) = -1} and \eqref{eq:t(n,n)+ with delta2}, 
\begin{align} \label{eq:l(i,q)a(i,q)}
    l_{iq} \, a_{iq} &= \sum_{j=2}^q l_{iq} \, l_{q j} \, (l_{i, \, j-1} \, t_{j-1, \, j} + l_{ij} \, t_{jj} + l_{i, \, j+1} \, t_{j+1, \, j}) \notag \\
    &=\left\{ \begin{aligned}
    &\sum_{j=2}^{q-1} \left( 2^{j-2} - \delta_{j-2} \right) - 2^{q-2} + \delta_{q-2}, \quad & q < i-1 \\
    &\sum_{j=2}^{q-1} \left( 2^{j-2} - \delta_{j-2} \right) - 2^{q-2} + \delta_{q-2} + 2 \, l_{q+1,\, q} \, t_{q+1,\, q}, \quad & q = i-1
    \end{aligned}
    \right. \notag \\
    &=\left\{ \begin{aligned}
    & -1 + \delta_{q-2} - \sum_{j=2}^{q-1} \delta_{j-2}, \quad & 2 \leqslant q < i-1 \leqslant n-2, \\
    & -1 + \delta_{q-2} - \sum_{j=2}^{q-1} \delta_{j-2} + 2 \, l_{q+1,\, q} \, t_{q+1,\, q}, \quad & 2 \leqslant q = i-1 \leqslant n-2.
    \end{aligned}
    \right.
\end{align}

Since $l_{iq} = \pm 1$ and $|a_{iq}| \leqslant 1$, then
$\left| -1 + \delta_{q-2} - \sum_{j=2}^{q-1} \delta_{j-2} \right| \leqslant 1$ produces the same condition as \eqref{ineq: bounds of all delta}. For the last equality of \eqref{eq:l(i,q)a(i,q)}, we have
\begin{equation} \label{ineq: bounds of l(q+1,q)t(q+1,q)}
- \frac{\delta_{q-2} - \sum_{j=0}^{q-3} \delta_{j}}{2} \leqslant l_{q+1,\, q} \, t_{q+1,\, q} \leqslant 1 - \frac{\delta_{q-2} - \sum_{j=0}^{q-3} \delta_{j}}{2}, \quad 2 \leqslant q \leqslant n-2.
\end{equation}

By \eqref{eq:l(q,j)l(i,j)l(i,q) = -1}, \eqref{eq:t(n,n)+ with delta2} and $k=i-1$, then \eqref{eq:a(k+1,k+1)}  satisfies
\begin{equation}
a_{ii} = \left\{ \begin{aligned}
& 4 \, l_{i,\,i-1} \, t_{i,\,i-1} + l_{i+1,\,i}\, t_{i+1,\,i} - \sum_{j=0}^{i-2} \, \left(2^j - \delta_j \right), \quad & 3 \leqslant i \leqslant n-2, \\
& 4 \, l_{i,\,i-1} \, t_{i,\,i-1} + \sum_{j=0}^{i-2} \, \delta_j - \delta_{i-1} + 1, \quad & i = n-1.
\end{aligned}
\right.
\end{equation}

Since $|a_{ii}| \leqslant 1$, then
\begin{equation} \label{ineq: bounds of l(q,q-1)t(q,q-1)l(q+1,q)t(q+1,q)}
2^{q-1} - 2 - \sum_{j=0}^{q-2} \, \delta_j \leqslant 4 \, l_{q,\,q-1} \, t_{q,\,q-1} + l_{q+1,\,q}\, t_{q+1,\,q} \leqslant 2^{q-1} - \sum_{j=0}^{q-2} \, \delta_j, \quad 3 \leqslant q \leqslant n-2,
\end{equation}
and
\begin{equation} \label{ineq: bounds of l(n-1,n-2)t(n-1,n-2)}
-\frac{1}{2} + \frac{\delta_{n-2}-\sum_{j=0}^{n-3} \, \delta_j}{4} \leqslant l_{n-1,\,n-2} \, t_{n-1,\,n-2} \leqslant \frac{\delta_{n-2}-\sum_{j=0}^{n-3} \, \delta_j}{4}.
\end{equation}

For $3 \leqslant q \leqslant n-2$, from \eqref{ineq: bounds of l(q+1,q)t(q+1,q)} and \eqref{ineq: bounds of l(q,q-1)t(q,q-1)l(q+1,q)t(q+1,q)}, 
\begin{equation} \label{ineq: bounds of l(q,q-1)t(q,q-1)}
\frac{2^{q} -6 - 3 \, \sum_{j=0}^{q-3} \, \delta_j - \delta_{q-2}}{8} \leqslant l_{q,\,q-1}\, t_{q,\,q-1} \leqslant \frac{2^{q} - 3 \, \sum_{j=0}^{q-3} \, \delta_j - \delta_{q-2}}{8}.
\end{equation}

So \eqref{ineq: bounds of l(q+1,q)t(q+1,q)} and \eqref{ineq: bounds of l(q,q-1)t(q,q-1)} are established simultaneously for $3 \leqslant q \leqslant n-2$ when
\begin{equation} \label{ineq: bounds of sum delta(j) - delta(q-3)+ delta(q-2)}
2^q-14 \leqslant 7 \, \sum_{j=0}^{q-4} \delta_{j} - \delta_{q-3} + \delta_{q-2} \leqslant 2^q, \qquad 3 \leqslant q \leqslant n-2.
\end{equation}

Let $q = n-2$ in \eqref{ineq: bounds of l(q+1,q)t(q+1,q)} and from \eqref{ineq: bounds of l(n-1,n-2)t(n-1,n-2)},
\begin{equation*}
    \left\{\begin{aligned}
    & -\frac{\delta_{n-4} - \sum_{j=0}^{n-5} \delta_{j}}{2} \leqslant \frac{\delta_{n-2}-\sum_{j=0}^{n-3} \, \delta_j}{4}, \\
    & -\frac{1}{2}+ \frac{\delta_{n-2}-\sum_{j=0}^{n-3} \, \delta_j}{4} \leqslant 1 -\frac{\delta_{n-4} - \sum_{j=0}^{n-5} \delta_{j}}{2}.
    \end{aligned}
    \right.
\end{equation*}

Hence,
\begin{equation} \label{ineq: bounds of two l(n-1,n-2)t(n-1,n-2)}
-6 \leqslant 3 \, \sum_{j=0}^{n-5} \delta_{j} - \delta_{n-4} + \delta_{n-3} - \delta_{n-2} \leqslant 0.
\end{equation}

Therefore plugging \eqref{ineq: bounds of all delta}, \eqref{ineq: bounds of sum delta(j) - delta(q-3)+ delta(q-2)} and \eqref{ineq: bounds of two l(n-1,n-2)t(n-1,n-2)} into the optimization problem \eqref{initial optimization problem},
\begin{equation} \label{optimization problem with more conditions}
\left\{
\begin{aligned}
& \delta = \min \, \sum_{j=0}^{n-2} \, \delta_j & \\
& \mbox{s. t. } \quad 0 \leqslant \delta_j \leqslant 2^{j+1},  & 0 \leqslant j \leqslant n-2, \\
& \quad \quad ~0 \leqslant \delta_{q-2} - \sum_{j=0}^{q-3} \delta_{j} \leqslant 2, & 3 \leqslant q \leqslant n-1, \\
& \quad \quad ~ 2^q-14 \leqslant 7 \, \sum_{j=0}^{q-4} \delta_{j} - \delta_{q-3} + \delta_{q-2} \leqslant 2^q, & 3 \leqslant q \leqslant n-2, \\
& \quad \quad ~ -6 \leqslant 3 \, \sum_{j=0}^{n-5} \delta_{j} - \delta_{n-4} + \delta_{n-3} - \delta_{n-2} \leqslant 0, & n \geqslant 4.
\end{aligned}
\right.
\end{equation}

Since $2^q -14 > 0$ ($q \geqslant 4$) and from \eqref{optimization problem with more conditions}, it is easy to obtain
\begin{equation}
\left\{\begin{aligned}
& \delta = 0, & n \leqslant 5, \\
& \delta > 0, & n \geqslant 6.
\end{aligned}
\right.
\end{equation}

Therefore, by \eqref{eq:t(n,n)+ with delta}, $t_{nn} < 2^{n-1}$ when $n \geqslant 6$. A similar argument shows that $t_{nn} > -2^{n-1}$ when $n \geqslant 6$. Hence, $| t_{nn} | < 2^{n-1}$ when $n \geqslant 6$, this  motivates the following Theorem.

\begin{theorem} \label{Th:growth factor bound for LTL^T factorization}
Let a symmetric indefinite matrix $A = (a_{ij})\in \mathbb{R}^{n \times n}$, then the growth factor bound for Aasen's algorithm is $2^{n-1}$, but that is not tight when $n \geqslant 6$.
\end{theorem}

We construct examples with $n = 4$ and $n=5$ such that the upper bounds are attainable. 
\begin{align*}
   A&=\begin{pmatrix}
      1 & 1 & -1 & 1 \\
      1 & \frac{\delta}{2} - 1 & 1 & \delta - 1 \\
     -1 & 1 & 1 & -1 \\
      1 & \delta - 1 & -1 & 1
   \end{pmatrix} \\
   &=
    \begin{pmatrix}
         1  &      &     &   \\
         0   &  1  &      &   \\
         0  &  -1  &   1  &    \\
         0  &   1  &   \mathbf{1}  &   1
    \end{pmatrix}
    \begin{pmatrix}
    1 & 1 &  &   \\
    1 & -1+\frac{\delta}{2} & \mathbf{\frac{\delta}{2}} & \\
     &  \mathbf{\frac{\delta}{2}} & 2+\frac{\delta}{2} &    -4 \\
     & & -4 & 8-2\delta
    \end{pmatrix}
    \begin{pmatrix}
    1  &   0   &  0   & 0  \\
            &  1  &   -1   & 1  \\
           &    &   1  &  \mathbf{1}  \\
           &     &     &   1
    \end{pmatrix},
\end{align*}
and
{\small
\begin{align*}
    A&=\begin{pmatrix}
      1 &  1 & 1 & 1 & -1 \\
      1 & \frac{\delta}{4} & 1 - \frac{\delta}{2} & \delta - 1 & 1 - \delta \\
      1 & 1 - \frac{\delta}{2} & 1 & 1 & 2\delta - 1 \\
      1 & \delta - 1 & 1 & 1 & -1 \\
     -1 & 1 - \delta & 2\delta - 1 &  -1 & 1
    \end{pmatrix} \\
    &=\begin{pmatrix}
    1 &  &   &   &   \\
    0 & 1 &  &   &   \\
    0 & 1 & 1 &  &   \\
    0 & 1 & -1 & 1 &  \\
    0 & -1 & 1 & \mathbf{1} &   1
    \end{pmatrix}
    \begin{pmatrix}
    1 & 1 & & & \\
    1 &  \frac{\delta}{4} & 1 & & \\
    & 1-\frac{3\delta}{4} & -1+\frac{5\delta}{4} & \mathbf{\delta} & \\
    & & \mathbf{\delta} & 4-\delta  & -8+4\delta \\
    & & & -8+4\delta & 16-12\delta
    \end{pmatrix}
    \begin{pmatrix}
    1 & 0 & 0 & 0 & 0 \\
     &  1 & 1 & 1 & -1   \\
     &   &  1 & -1 & 1 \\
     &   &  & 1 & \mathbf{1} \\
     &  &    &  &   1
    \end{pmatrix}.
\end{align*}
}

When $\delta \to 0^+$ in these two matrices $A$, the upper bounds are attainable. 
For the following $n=6$ example,

\begin{align*}
A&=\begin{pmatrix}
     1 &  1 &  1 & 1 & 1 & -1 \\
 1 & \delta/2 - 3/4 & -1/2 & \delta - 1 & \delta - 1 & 1 - \delta \\
 1 & -1/2 & -1 & -1 & 1 & -1 \\
 1 & \delta - 1 & -1 & 5 \delta - 3 & 1 & 2 \delta - 1 \\
 1 & \delta - 1 & 1 & 1 & 1 & -1 \\
-1 & 1 - \delta & -1 & 2 \delta - 1 & -1 & 1
\end{pmatrix} \\
&=\begin{pmatrix}
    1 &  &   &   & &     \\
    0 & 1 &  &   &  &  \\
    0 & 1 & 1 &  &   &  \\
    0 & 1 & -1 & 1 & &  \\
    0 & 1 & -1 & -1 & 1 & \\
    0 & -1 & 1 & 1 & 1 & 1
    \end{pmatrix}
    \begin{pmatrix}
    1 & 1 & & & & \\
    1 &  -\frac{3}{4}+\frac{1}{2}\delta & \frac{1}{4}-\frac{1}{2}\delta & & & \\
    & \frac{1}{4}-\frac{1}{2}\delta & -\frac{3}{4}+\frac{1}{2}\delta & -1 & & \\
    & & -1 & -3+3 \delta  & \delta & \\
    & & & \delta & 8-3\delta & -16+8\delta \\
    & & & & -16+8\delta & 32-20 \delta
    \end{pmatrix}\\
    &\qquad \begin{pmatrix}
    1 &  &   &   & &     \\
    0 & 1 &  &   &  &  \\
    0 & 1 & 1 &  &   &  \\
    0 & 1 & -1 & 1 & &  \\
    0 & 1 & -1 & -1 & 1 & \\
    0 & -1 & 1 & 1 & 1 & 1
    \end{pmatrix}^T.
\end{align*}

Since $\left| a_{ij} \right| \leq 1$, then $\frac{2}{5}
\leq \delta \leq \frac{4}{5}$ and the growth factor is $24<2^5$ when $n=6$.

\section{Conclusion} 
\label{sec:conclusion}
Variants of Aasen's algorithm have been incorporated into LAPACK, as they can be effectively applied to solve symmetric indefinite linear systems.
We estimate the growth factor upper bound for Aasen's algorithm by direct computation and obtain the value $2^{n-1}$, which is much smaller than that given by Higham. We also prove that the bound is not tight when the matrix size is greater than $6$. Besides, some matrix examples are constructed to verify our theoretical analysis with $n=4,~5,~6$.

\bibliographystyle{unsrt}
\bibliography{rcp.bib}

\end{document}